\documentclass[12pt]{amsart}

\usepackage[T1]{fontenc}
\usepackage[utf8]{inputenc}
\usepackage{amsmath,amssymb,amsthm}
\usepackage{hyperref}
\usepackage[capitalize]{cleveref}
\usepackage{stmaryrd}

\makeatletter
\@namedef{subjclassname@2020}{%
  \textup{2020} Mathematics Subject Classification}
\makeatother

\newcommand{\Q}{\mathbb{Q}}
\newcommand{\Z}{\mathbb{Z}}
\newcommand{\V}{\mathcal{V}}
\newcommand{\Proj}{\mathbb{P}}
\newcommand{\p}{\mathfrak{p}}
\newcommand{\R}{\mathbb{R}}
\newcommand{\N}{\mathbb{N}}

\newcommand{\eps}{\varepsilon}

\newcommand{\st}{\colon}

\newcommand{\abs}[1]{\left|#1\right|}
\newcommand{\Zint}[1]{\left\llbracket#1\right\rrbracket}
\newcommand{\arch}[1]{\V_{#1}^\infty}
\newcommand{\narch}[1]{\V_{#1}^0}
\newcommand{\places}[1]{\V_{#1}}
\newcommand{\primes}[1]{\mathcal{P}_{#1}}
\newcommand{\nzero}{\backslash\{0\}}
\newcommand{\ceil}[1]{\left\lceil#1\right\rceil}
\newcommand{\norm}{N_{L/\Q}}

\DeclareMathOperator{\h}{\mathit{h}}

\DeclareMathOperator{\hmod}{\smash{\widetilde{\mathit{h}}}}
\DeclareMathOperator{\Log}{Log}
\DeclareMathOperator{\bigH}{\mathit{H}}

\newtheorem{prop}{Proposition}[section]
\crefname{prop}{Proposition}{Propositions}
\newtheorem{cor}[prop]{Corollary}
\newtheorem{lem}[prop]{Lemma}
\newtheorem{thm}[prop]{Theorem}
\crefname{lem}{Lemma}{Lemmas}
\theoremstyle{definition}
\newtheorem{defn}[prop]{Definition}
\newtheorem{claim}[prop]{Claim}
\newtheorem{rem}[prop]{Remark}
\crefname{rem}{Remark}{Remarks}

\author{Jean Kieffer}
\address{Harvard University, Mathematics Department\\
  Cambdidge, MA 02138, United States}
\email{kieffer@math.harvard.edu}
\date{}

\title[Heights of fractions from their values] {Upper bounds on the
  heights of polynomials and rational fractions from their values}

\begin{document}

%%%%% To ease editing, for IMPAN journals add:

\baselineskip=17pt

\begin{abstract}
  Let~$F$ be a univariate polynomial or rational fraction of
  degree~$d$ defined over a number field. We give bounds from above on
  the absolute logarithmic Weil height of~$F$ in terms of the heights
  of its values at small integers: we review well-known bounds
  obtained from interpolation algorithms given values at~$d+1$
  (resp.~$2d+1$) points, and obtain tighter results when considering a
  larger number of evaluation points.
\end{abstract}

\subjclass[2020]{11C08, 11G50}

\keywords{Heights, polynomials, rational fractions}

\maketitle

\section{Introduction}
\label{sec:intro}

Let~$F$ be a univariate rational fraction of degree~$d$ defined
over~$\Q$. The \emph{height} of~$F$, denoted by~$\h(F)$, measures the
size of the coefficients of~$F$. To define it, write $F=P/Q$ where
$P,Q\in \Z[X]$ are coprime; then $\h(F)$ is the maximum value
of~$\log\abs{c}$, where~$c$ runs through the nonzero coefficients
of~$P$ and~$Q$. In particular, if~$x=p/q$ is a rational number in
irreducible form, then~$\h(x) = \log\max\{\abs{p},\abs{q}\}$.

Heights can be generalized to arbitrary number fields, and are a basic
tool in diophantine geometry
\cite[Part~B]{hindry_DiophantineGeometry2000}. They are also
meaningful from an algorithmic point of view: the amount of memory
needed to store~$F$ in a computer is in general $O(d \h(F))$, and the
cost of manipulating~$F$ grows with the size of its coefficients.

In this paper, we are interested in the relation between the height
of~$F$ and the heights of evaluations~$F(x)$, where~$x$ is an
integer. One direction is easy: by
\cite[Prop.~B.7.1]{hindry_DiophantineGeometry2000}, we have
\begin{equation}
  \label{eq:height-eval}
  \h(F(x)) \leq d \h(x) + \h(F) + \log(d+1).
\end{equation}
In the other direction, when we want to bound~$\h(F)$ from the heights
of its values, matters are more complicated.

An easy case is when~$F\in\Z[X]$ is a polynomial with integer
coefficients of degree at most~$d\geq 1$. Then, looking at the
archimedean absolute value of the coefficients of~$F$ is sufficient to
bound~$\h(F)$. Moreover, given height bounds on~$d+1$ values of~$F$,
the Lagrange interpolation formula allows us to bound~$\h(F)$ in a
satisfactory way. For instance, assuming that
\begin{displaymath}
  \h(F(i)) \leq H \quad \text{for every } 0\leq i\leq d,
\end{displaymath}
we easily obtain
\begin{displaymath}
  \h(F) \leq H + d\log(2d) + \log(d+1).
\end{displaymath}
This result can be refined and adapted to other sets of intepolation
points \cite[Lem.~20]{broker_ExplicitHeightBound2010},
\cite[Lem.~4.1]{pazuki_ModularInvariantsIsogenies2019}; in any case
the bound on~$\h(F)$ is roughly~$H$ up to additional terms
in~$O(d\log d)$. This is consistent with
inequality~\eqref{eq:height-eval}.

When~$F$ is a rational fraction or even a polynomial with rational
coefficients, this result breaks down, and surprisingly little
information appears in the literature despite the simplicity of the
question.

\subsection{Polynomials.}
Let us first consider the case where~$F$ is a polynomial in~$\Q[X]$,
of degree at most~$d\geq 1$. Then~$F$ is determined by its values
at~$d+1$ distinct points. Let $x_1,\ldots, x_{d+1}$ be distinct
integers, let~$H\geq 1$, and assume that $\h(F(x_i))\leq H$ for
every~$i$. This time, the Lagrange interpolation formula yields a
bound on~$\h(F)$ which is roughly~$O(dH)$ (see
\cref{prop:poly-basic}). This is intuitive enough: in general,
computing~$F$ from its values~$F(x_i)$ involves reducing the rational
numbers~$F(x_i)$ to the same denominator, thus multiplying the heights
of the input by the number of evaluation points. But then,
inequality~\eqref{eq:height-eval} is very pessimistic at each of the
evaluation points~$x_i$: massive cancellations occur with the
denominator of~$F$, and the height of~$F(x_i)$ is just a fraction
$1/d$ of the expected value.

However, if we consider \emph{more} than~$d+1$ evaluation points
$x_1\ldots,x_N$ such that~$\h(F(x_i))\leq H$, we will likely find an
evaluation point where inequality~\eqref{eq:height-eval} is accurate,
and hence obtain a bound on~$\h(F)$ of the form~$O(H)$ rather
than~$O(dH)$. We prove the following result in this direction.

\newcommand{\polystatement}{
  Let~$L$ be a number field, and let $\Zint{A,B}$ be an interval
  in~$\Z$.  Write $D=B-A$ and $M=\max\{\abs{A},\abs{B}\}$.
  Let~$F\in L[X]$ be a polynomial of degree at most~$d\geq 1$,
  let~$N\geq d+1$, and let $x_1,\ldots,x_N$ be distinct elements
  of~$\Zint{A,B}$. Assume that $\h(F(x_i))\leq H$ for every
  $1\leq i\leq N$. Then we have
  \begin{displaymath}
    \h(F) \leq \frac{N}{N-d} H + D\log(D) + d\log(2M) + \log(d+1).
  \end{displaymath}
}

\begin{thm}
  \label{thm:main-poly}
  \polystatement
\end{thm}

For instance, we obtain a bound on~$\h(F)$ which is linear in~$H$ when
considering~$N=2d$ evaluation points. See also
\cref{thm:main-poly-proved} for local versions of this result.

\subsection{Rational fractions.}
Second, consider the case where $F\in \Q(X)$ is a rational fraction of
degree at most~$d\geq 1$. Then~$F$ is determined by its values
at~$2d+1$ points. If~$x_1,\ldots, x_{2d+1}$ are distinct integers
which are not poles of~$F$, and if $\h(F(x_i))\leq H$ for every~$i$,
then a direct analysis of the interpolation algorithm yields a bound
on~$\h(F)$ which is roughly~$O(d^2H)$ (see
\cref{prop:frac-basic-L}). As above, we can ask for a bound which is
linear in~$H$ when more evaluation points are given.

In this case we could imagine cases where~$F=P/Q$ has a very large
height, but massive cancellations happen in many quotients
$P(x_i)/Q(x_i)$. This makes the result more intricate.

\newcommand{\mainstatement}{Let~$L$ be a number field of degree~$d_L$
  over~$\Q$ and discriminant~$\Delta_L$. Let $\Zint{A,B}$ be an
  interval in~$\Z$, and write $D=B-A$ and $M=\max\{\abs{A},\abs{B}\}$.
  Let \mbox{$F\in L(X)$} be a univariate rational fraction of degree
  at most~$d\geq 1$. Let~$S$ be a subset of~$\Zint{A,B}$ which
  contains no poles of~$F$, let~$\eta\geq 1$, and let
  \mbox{$H\geq \max\{4, \log(2M)\}$}.  Assume that
  \begin{enumerate}
  \item $\h(F(x))\leq H$ for every $x\in S$.
  \item $S$ contains at least $D/\eta$ elements.
  \item $D \geq \max\{\eta d^3 H, 4 \eta d d_L\}$.
  \end{enumerate}
  Then we have
  \begin{displaymath}
    \h(F) \leq H + C_L\eta d\log (\eta dH) + d \log(2M) + \log(d+1),
  \end{displaymath}
  where~$C_L$ is a constant depending only on~$d_L$ and~$\Delta_L$. We
  can take $C_\Q=960$.
}

\begin{thm}
  \label{thm:main-frac}
  \mainstatement
\end{thm}

We can give a general explicit expression for the constant~$C_L$ in
terms of~$d_L$ and~$\Delta_L$ (see~§\ref{sec:frac-proof}).  The number
of evaluation points needed in this result is quite large, and depends
on~$H$.  Still, \cref{thm:main-frac} is strong enough to imply the
following result.

\newcommand{\corstatement}{
  Let~$c\geq 1$, and let~$F\in\Q(X)$ be a rational fraction of degree at
  most~$d\geq 1$. Let~$V\subset\Z$ be a finite set such that~$F$ has
  no poles in $\Z\backslash V$. Assume that for
  every~$x\in \Z\backslash V$, we have
  \begin{displaymath}
    \h(F(x))\leq c\max\{1, d\log d + d\h(x)\}.
  \end{displaymath}
  Then there exists a constant~$C = C(c, \#V)$ such that
  \begin{displaymath}
    \h(F) \leq C d\log(4d).
  \end{displaymath}
  Explicitly, we can take
  $C = (4c+1923) (12 + \log\max\{1,\#V\} + 2\log(c))$.  }

\begin{cor}
  \label{cor:main}
  \corstatement
\end{cor}

It would be interesting to know whether we can obtain an efficient
bound on~$\h(F)$ using only~$O(d)$ evaluation points, as was the case
for polynomials, instead of~$O(d^3H)$. The constants in
\cref{thm:main-frac,cor:main} are not optimal; smaller constants can
be obtained following the same proofs, at the cost of lengthier
expressions.

The author has applied these results to obtain tight asymptotic height
bounds for modular equations on PEL Shimura
varieties~\cite{kieffer_DegreeHeightEstimates2020}, for instance
modular equations of Siegel and Hilbert type for abelian surfaces,
generalizing existing works in the case of classical modular
polynomials~\cite{pazuki_ModularInvariantsIsogenies2019}. These
modular equations are examples of rational fractions whose evaluations
can be shown to have small height.

\subsection*{Organization of the paper.}
In \cref{sec:heights}, we recall the definition of heights over a
number field that we use in the whole paper. In~\cref{sec:poly}, we
prove~\cref{thm:main-poly} about the heights of polynomials. To
prepare for the case of rational fractions, we study the relations
between heights and norms of integers in number fields
in~\cref{sec:heights-norms}. We prove height bounds for rational
fractions using the minimal number of evaluation points
in~\cref{sec:frac-basic}. Finally, \cref{sec:frac-lem,sec:frac-proof}
are devoted to the proof of~\cref{thm:main-frac}.

\subsection*{Acknowledgements} Thanks are due to the anonymous referee
for pointing out several errors in an earlier version of this
paper. This work is part of the author's PhD dissertation at the
University of Bordeaux (France), and he warmly thanks Damien Robert
and Aurel Page for their advice and encouragement.

\section{Heights over number fields}
\label{sec:heights}

Let~$L$ be a number field of degree~$d_L$ over~$\Q$. Write $\narch{L}$
(resp.~$\arch{L}$) for the set of all nonarchimedean
(resp.~archimedean) places of~$L$, and
write~\mbox{$\places{L} = \narch{L}\sqcup
  \arch{L}$}. Let~$\mathcal{P}_\Q$ (resp.~$\mathcal{P}_L$) be the set
of primes in~$\Z$ (resp.~prime ideals in the ring of integers~$\Z_L$
of~$L$).

For each place~$v$ of $L$, the local degree of~$L/\Q$ at~$v$ is
$d_v = [L_v:\Q_v]$, where subscripts denote completion.  Denote
by~$\abs{\cdot}_v$ the normalized absolute value associated with~$v$:
when $v\in \narch{L}$, and~$p\in \primes{\Q}$ is the prime below~$v$,
we have $\abs{p}_v = 1/p$. When~$v$ is archimedean, $\abs{\cdot}_v$ is
the usual real or complex absolute value.

The absolute logarithmic Weil height of projective tuples, affine
tuples, polynomials and rational fractions over~$L$ is defined as
follows \cite[§B.2 and §B.7]{hindry_DiophantineGeometry2000}.

\begin{defn} Let~$n\geq 1$, and let~$a_0,\ldots, a_n\in L$.
  \label{def:heights}
  \begin{enumerate}
  \item If the~$a_i$ are not all zero, the projective height
    of~$(a_0:\cdots:a_n)\in \Proj^n_L$ is
    \begin{displaymath}
      \h_{\mathrm{proj}}(a_0:\ldots:a_n) = \sum_{v\in\places{L}} \frac{d_v}{d_L} \log
      \Bigl(\max_{0\leq i\leq n} \abs{a_i}_v \Bigr).
    \end{displaymath}
  \item The affine height of~$(a_1,\ldots,a_n)\in L^n$ is the
    projective height of the tuple~$(1:a_1:\cdots:a_n)$:
    \begin{displaymath}
      \h(a_1,\ldots,a_n) = \sum_{v\in\places{L}} \frac{d_v}{d_L}
      \log \bigl(\max\{1, \max_{1\leq i\leq n} \abs{a_i}_v\} \bigr).
    \end{displaymath}
    In particular, for~$a\in L$, we have
    \begin{displaymath}
      \h(a) = \h_{\mathrm{proj}}(1:a) = \sum_{v\in\places{L}}
      \dfrac{d_v}{d_L} \log \bigl( \max\{1, \abs{a}_v\} \bigr).
    \end{displaymath}
  \item Let~$P = \sum_{i=0}^n a_i X^i \in L[X]$. For every
    place~$v\in\places{L}$, we write
    \begin{displaymath}
      \abs{P}_v = \max_i \abs{a_i}_v.
    \end{displaymath}
    The height of~$P$ is defined as the affine height of~$(a_0,\ldots,a_n)$. In other words
    \begin{displaymath}
      \h(P) = \sum_{v\in\places{L}} \dfrac{d_v}{d_L} \log
      \bigl(\max\{1, \abs{P}_v\} \bigr).
    \end{displaymath}
    If $\p\in\primes{L}$ is a prime ideal, we also define
    the~$\p$-adic valuation of~$P$ as
    \begin{displaymath}
      v_\p(P) = \min_{0\leq i\leq n} v_\p(a_i).
    \end{displaymath}
  \item Finally, if~$F\in L(X)$ is a rational fraction, and~$F = P/Q$
    where \mbox{$P,Q\in L[X]$} are coprime, we define~$\h(F)$ as the height
    of the projective tuple formed by all the coefficients of~$P$
    and~$Q$.
  \end{enumerate}
\end{defn}

If~$L=\Q$, then \cref{def:heights} coincides with the naive definition
of heights given in the introduction. By the product formula, heights
are independent of the ambient field
\cite[Lem.~B.2.1(c)]{hindry_DiophantineGeometry2000}. Recall that
\begin{equation}
  \label{eq:sum-degrees}
  \sum_{v\in\arch{L}}\frac{d_v}{d_L} = 1,
\end{equation}
a fact we will use many times when computing archimedean parts of heights.
Moreover, if~$x,y,z\in L$ with~$z\neq 0$, then we have
\begin{equation}
  \label{eq:height-prod-inv}
  \h(xy)\leq \h(x)+\h(y)\quad \text{and}\quad
  \h(1/z)=\h(z).
\end{equation}

As \cref{def:heights} suggests, in order to obtain height bounds for
polynomials and rational fractions, we will try to bound their
coefficients from above in the absolute values associated with all the
places of~$L$.

\section{Heights of polynomials from their values}
\label{sec:poly}

In this section, we estimate the height of a polynomial~$F\in L[X]$ of
degree at most~$d\geq 1$ in terms of the heights of evaluations
of~$F$. We choose our evaluation points to be integers in an
interval~$\Zint{A,B}\subset\Z$, and we write~\mbox{$D=B-A$} and
$M=\max\{\abs{A},\abs{B}\}$ (here $\abs{\cdot} = \abs{\cdot}_\infty$
is the archi\-medean absolute value). Our tool is the Lagrange
interpolation formula: if $x_1,\ldots, x_{d+1}\in\Zint{A,B}$ are
distinct, then
\begin{equation}
  \label{eq:lagrange}
  F = \frac{1}{D!} \sum_{i=1}^{d+1} F(x_i) Q_i
  \qquad \text{where } Q_i = D! \frac{\prod_{j\neq i}(X-x_j)}{\prod_{j\neq i}(x_i-x_j)} \in\Z[X].
\end{equation}

\begin{lem}
  \label{lem:lagrange-coefficients}
  In the notation of equality~\eqref{eq:lagrange}, we have
  $\abs{Q_i}_\infty \leq D!\, (2M)^d$ for all~$1\leq i\leq d+1$.
\end{lem}

\begin{proof}
  Since the denominator~$\prod_{j\neq i} (x_i - x_j)$ divides~$D!$, we
  have 
  \begin{displaymath}
    Q_i = N_i\prod_{j\neq i} (X - x_j)
  \end{displaymath}
  for some $N_i\in\Z$ dividing~$D!$. Therefore, for every
  $0\leq k\leq d$, if~$c_k$ denotes the coefficient of~$X^{d-k}$
  in~$Q_i$, we have
  \begin{displaymath}
    \abs{c_k}_\infty\leq \abs{N_i}_\infty \binom{d}{k} M^k \leq D!\, 2^d M^d. \qedhere
  \end{displaymath}
\end{proof}

A straightforward application of the Lagrange formula on~$d+1$
evaluation points yields the following result.

\begin{prop}
  \label{prop:poly-basic}
  Let $F\in L[X]$ be a univariate polynomial of degree at
  most~$d\geq 1$, and let $x_1,\dots, x_{d+1}$ be distinct integers in
  $\Zint{ A, B}$. Write \mbox{$D=B-A$} and~$M=\max\{\abs{A},\abs{B}\}$.
  \begin{enumerate}
  \item \label{it:poly-local} For every~$v\in \narch{L}$, we have
    \begin{displaymath}
      \abs{F}_v \leq \abs{\dfrac{1}{D!}}_v\max\{ \abs{F(x_1)}_v,\ldots,
      \abs{F(x_{d+1})}_v\},
    \end{displaymath}
    and for every~$v\in \arch{L}$, we have
    \begin{displaymath}
      \abs{F}_v \leq (d+1)(2M)^d \max\{ \abs{F(x_1)}_v,\ldots,
      \abs{F(x_{d+1})}_v\}.
    \end{displaymath}
  \item \label{it:poly-global} Assume that
    $\h(F(x_i))\leq H$ for every~$1\leq i\leq d+1$. Then
    \begin{displaymath}
      \h(F) \leq (d+1)H + D\log(D) +
      d\log(2M) + \log(d+1).
    \end{displaymath}
  \end{enumerate}
\end{prop}

\begin{proof} Part~\ref{it:poly-local} is an immediate consequence of
  the interpolation formula~\eqref{eq:lagrange}, and
  \cref{lem:lagrange-coefficients} for archimedean places. For
  part~\ref{it:poly-global}, let~$v$ be a place of~$L$. By
  part~\ref{it:poly-local}, we have
  \begin{displaymath}
    \max\{1, \abs{F}_v\}\leq C_v \prod_{i=1}^{d+1} \max\{1, \abs{F(x_i)}_v\}
  \end{displaymath}
  where~$C_v = \abs{1/D!}_v$ if~$v$ is nonarchimedean, and
  $C_v = (d+1)(2M)^d$ if~$v$ is archimedean. Taking logarithms and summing,
  we obtain
  \begin{displaymath}
    \h(F) \leq \h(1/D!) + \bigl(d\log(2M)+\log(d+1)\bigr) \sum_{v\in \arch{L}} \dfrac{d_v}{d_L}
    + \sum_{i=1}^{d+1} \h(F(x_i)).
  \end{displaymath}
  By~eq.~\eqref{eq:height-prod-inv}, we
  have~$\h(1/D!) = \h(D!) = \log(D!)\leq D\log(D)$. The result follows
  then from~eq.~\eqref{eq:sum-degrees}.
\end{proof}

It is interesting to compare \cref{prop:poly-basic} with
\cite[Cor.~B.2.6]{hindry_DiophantineGeometry2000}, using the
evaluation maps at~$x_i$ as linear maps from~$L[X]$ to~$L$: under the
hypotheses of the proposition, the height of the \emph{tuple}
$(F(x_1),\ldots,F(x_{d+1}))$ can be as large as~$(d+1)H$.

\begin{rem}
  The result of \cref{prop:poly-basic} takes a particularly nice form
  because the evaluation points~$x_i$ are integers taken in a fixed
  interval. If we only assume the~$x_i$ to be distinct algebraic
  integers of bounded height, then providing an upper bound on the
  height of a common multiple of all products of the
  form~$\prod_{j\neq i} (x_i - x_j)$ seems more complicated. A similar
  issue arises when the~$x_i$ are only assumed to be distinct points
  in~$\Q$ of bounded height.  However, if the evaluation points~$x_i$
  are chosen to be rational numbers with the same denominator, then
  one can still apply \cref{prop:poly-basic} to a rescaled
  polynomial. In the rest of this paper, we will continue to consider
  (almost) consecutive integers as evaluation points.
\end{rem}

Better upper bounds on~$\h(F)$ can be obtained given height bounds on
more than~$d+1$ values of~$F$: this is the content of
\cref{thm:main-poly}, which we recall here with additional local
statements.

\begin{thm}
  \label{thm:main-poly-proved}
  \polystatement
  More precisely, for every~$v\in \places{L}$, we have
  \begin{displaymath}
    \log \max\{1,\abs{F}_v\} \leq C_v
    + \dfrac{1}{N-d} \sum_{i=1}^{N} \log \max\{1, \abs{F(x_i)}_v\}
  \end{displaymath}
  where~$C_v = \log\abs{1/D!}_v$ if~$v\in \narch{L}$, and
  $C_v = d\log (2M) + \log (d+1)$ if~$v\in \arch{L}$.
\end{thm}

We will need the following lemma.

\begin{lem}
  \label{lem:bad-values}
  Keep the notation from \cref{thm:main-poly-proved}, and let
  $v\in\narch{L}$ (resp.\ $v\in\arch{L})$. Then the number of
  elements~$x\in \Zint{ A,B}$ satisfying the inequality
  \begin{displaymath}
    \abs{F(x)}_{v} < \abs{D!\,F}_{v} \quad \Bigl( \text{resp.\ } \abs{F(x)}_{v} <
    \dfrac{\abs{F}_v}{(2M)^d (d+1)} \Bigr)
  \end{displaymath}
  is at most~$d$.
\end{lem}

\begin{proof}[Proof of \cref{lem:bad-values}]
  We argue by contradiction, using part~\ref{it:poly-local} of
  \cref{prop:poly-basic}.
\end{proof}

\begin{proof}[Proof of \cref{thm:main-poly-proved}] It is enough to
  prove the local statements: after that, the global statement results
  from summing all the local contributions. Let~$v$ be a place of~$L$.
  If $v\in\narch{L}$, then by \cref{lem:bad-values}, we have
  \mbox{$\abs{F(x_i)}_v \geq \abs{D!\,F}_v$} for at least~$N-d$ values
  of~$i$. Therefore,
  \begin{displaymath}
    \prod_{i = 1}^{N} \max\{1, \abs{F(x_i)}_v\} \geq \abs{D!\,F}_v^{N-d}
  \end{displaymath}
  and
  \begin{displaymath}
    \log \max\{1,\abs{F}_v\} \leq \log\abs{\dfrac{1}{D!}}_v
    + \dfrac{1}{N-d} \sum_{i=1}^{N} \log \max\{1, \abs{F(x_i)}_v\}.
  \end{displaymath}
  Similarly, if $v\in\arch{L}$, then at least~$N-d$ of the~$F(x_i)$ satisfy the
  inequality \mbox{$\abs{F(x_i)}_v \geq \abs{F}_v/(2M)^d (d+1)$}, so
  \begin{align*}
    \log \max\{1,\abs{F}_v\} 
    &\leq d\log (2M) +
           \log (d+1) \\
    &\quad + \dfrac{1}{N-d}\sum_{i=1}^{N} \log\max\{1,\abs{F(x_i)}_v\}. \qedhere
  \end{align*}
\end{proof}

\section{Heights and norms of integers}
\label{sec:heights-norms}

Let~$L$ be a number field, let~$\Z_L$ be its ring of integers, and
let~$\Delta_L$ be its discriminant.  In this section, we study the
relation between the height of elements of~$\Z_L$ and their norms. We
denote the norm of elements and fractional ideals in~$L$ by~$\norm$.

\begin{defn}
  Let~$x\in L\nzero$. Then we define
  \begin{displaymath}
    \hmod(x) = \frac{1}{d_L} \log\abs{\norm(x)}
    = \sum_{v\in\arch{L}} \frac{d_v}{d_L} \log\abs{x}_v.
  \end{displaymath}
  If~$\mathfrak a$ is a fractional ideal in~$L$, we also write
  \begin{displaymath}
    \hmod(\mathfrak a) = \frac{1}{d_L} \log \norm(\mathfrak a).
  \end{displaymath}
\end{defn}

If the reader is interested in the case~$L=\Q$, then the remainder of
this section can be safely skipped since~$\hmod$ and~$\h$ are equal
on~$\Z$. In general, they are not equal: for instance,~$\hmod$ is
invariant under multiplication by units. This is not the case for~$\h$
as soon as~$L$ admits a fundamental unit, by the Northcott property
\cite[Thm.~B.2.3]{hindry_DiophantineGeometry2000}.

\begin{lem}
  \label{lem:hmod-h}
  Let $x\in\Z_L\nzero$. Then we have
  \begin{displaymath}
    0\leq \hmod(x) \leq \h(x).
  \end{displaymath}
  Equality holds on the right if and only if $\abs{x}_v\geq 1$ for every
  $v\in \arch{L}$.
\end{lem}

\begin{proof}
  We have $N_{L/\Q}(c)\in\Z\backslash\{0\}$, hence
  $\abs{N_{L/\Q}(c)}\geq 1$ and~\mbox{$\hmod(x)\geq 0$}. The rest is
  obvious.
\end{proof}

\begin{prop}
  \label{prop:height-reduce}
  There exists a constant~$C$ depending only on~$L$ such that for
  every $x\in\Z_L\nzero$, there exists a unit $\eps\in \Z_L^\times$
  such that
  \begin{displaymath}
    \h(\eps x) \leq \max \{ C
    ,\,\hmod(x)\}.
  \end{displaymath}
  We can take $C = d_L \sum_{i\in I} \h(\eps_i)$, where
  $(\eps_i)_{i\in I}$ is any basis of units in~$\Z_L$. 
\end{prop}

\begin{proof}
  Let $m = \#\arch{L}$.  In~$\R^m$, we define the
  hyperplane~$H_s$ for $s\in \R$ as follows:
  \begin{displaymath}
    H_s = \{(t_1,\ldots,t_m)\in\R^m \st t_1+\cdots+t_m
    = s\}.
  \end{displaymath}
  We also define the convex cone~$\Delta_s$ as follows:
  \begin{displaymath}
    \Delta_s = \bigl\{(t_1,\ldots,t_m)\in\R^m\st  \forall
    i,\  t_i\geq - s\bigr\}.
  \end{displaymath}
  The image of~$\Z_L^\times$ under the logarithmic embedding
  \begin{displaymath}
    \Log = \Bigl(\dfrac{d_v}{d_L} \log |\cdot|_{v} \Bigr)_{v\in\arch{L}}
  \end{displaymath}
  is a full rank lattice~$\Lambda$ in~$H_0$.
  Let~$(\eps_i)_{1\leq i\leq m-1}$ be a basis of units in~$\Z_L$, and
  let~$V$ be the following fundamental cell of~$\Lambda$:
  \begin{displaymath}
    V = \Bigl\{\sum_{i=1}^{m-1} \lambda_i \Log(\eps_i) \st \lambda_i\in [-\tfrac12,\tfrac12] \text{ for all } i\Bigr\}.
  \end{displaymath}
  For each~$v\in \arch{L}$ and each $1\leq i\leq m-1$, we have
  \begin{displaymath}
    \dfrac{d_v}{d_L} \log\abs{\eps_i}_v \geq - \dfrac{d_v}{d_L}\log\max\{1, \abs{1/\eps_i}_v\} \geq -\h(1/\eps_i) = -h(\eps_i).
  \end{displaymath}
  Therefore~$V$ is included in~$H_0\cap \Delta_s$ for every
  $s\geq s_{\min} = \frac{1}{2}\sum_{i=1}^{m-1} \h(\eps_i)$. From this,
  we deduce:
  \begin{enumerate}
  \item \label{it:h-hmod-1} For every $s\geq ms_{\min}$, the set
    $H_{s}\cap \Delta_0$ contains a translate of~$V$; indeed its
    translate by $-s/m\cdot(1,\ldots,1)$ is
    $H_0 \cap \Delta_{s/m}$.
  \item \label{it:h-hmod-2} For every $s\geq 0$, the set
    $H_s\cap \Delta_{s_{\min}}$ contains a translate of~$V$; indeed
    its translate by $-s/m\cdot(1,\ldots,1)$ is
    $H_0\cap \Delta_{s_{\min} + s/m}$.
  \end{enumerate}

  Let $x\in\Z_L\backslash\{0\}$, and consider the point
  \begin{displaymath}
    \Log(x) = \Bigl(\dfrac{d_v}{d_L} \log |x|_{v}\Bigr)_{v\in\arch{L}} \in\R^m.
  \end{displaymath}
  The sum of its coordinates is $s_x = \hmod(x)$.  If
  $s_x\geq m s_{\min}$, then by~\eqref{it:h-hmod-1} there exists a
  unit~$\eps\in\Z_L^\times$ such that $\Log(x) + \Log(\eps)$ belongs
  to~$\Delta_0$. Then $\abs{\eps x}_{v}\geq 1$ for every
  $v\in\arch{L}$, so
  \begin{displaymath}
    \h(\eps x) = \hmod(\eps x) = \hmod(x)
  \end{displaymath}
  by \cref{lem:hmod-h}.
  
  On the other hand, if $0\leq s_x < m s_{\min}$, then
  by~\eqref{it:h-hmod-2} we can still find a unit~$\eps$ such that
  $\Log(x)+\Log(\eps) \in \Delta_{s_{\min}}$, in other words
  \begin{displaymath}
    \frac{d_v}{d_L} \log |\eps x|_v \geq -s_{\min}
  \end{displaymath}
  for all $v\in\arch{L}$. Then
  \begin{displaymath}
    \h(\eps x) = \sum_{v\in\arch{L}} \dfrac{d_v}{d_L} \log\max\{1,|\eps x|_v\} \leq
    \hmod(\eps x) + \sum_{v\in\V_\infty} s_{\min} \leq 2m s_{\min}.
  \end{displaymath}
  This proves the proposition with $C = 2m s_{\min} \leq 2d_L s_{\min}$.
\end{proof}

\begin{rem}
  \label{rem:CL1}
  We can give an explicit upper bound for an acceptable constant~$C$
  in \cref{prop:height-reduce} in terms of the degree and discriminant
  of~$L$ only. Let~$\mathfrak R_L$ be the regulator
  of~$L$. By~\cite[Lem.~1]{bugeaud_BoundsSolutionsUnit1996},~$L$
  admits a basis of units~$(\eps_i)_{1\leq i\leq m-1}$ (where
  $m = \#\arch{L}$) such that
  \begin{displaymath}
    \h(\eps_i) \leq \dfrac{((m-1)!)^2}{2^{m-1}d_L^{m-1}}
    \Bigl(\frac{\delta(L)}{d_L} \Bigr)^{2-m} \mathfrak R_L
  \end{displaymath}
  for each $1\leq i\leq m-1$; here $\delta(L)>0$ satisfies the
  property that all non-roots of unity in~$L$ have height at least
  $\delta(L)/d_L$. It is known that we can take
  $\delta(L) = \log(2)/d_L$ if $d_L\leq 2$, and
  \begin{displaymath}
    \delta(d_L) = \max\Bigl\{\dfrac{1}{53d_L\log(6d_L)},
    \frac14 \Bigl(\frac{\log\log d_L}{\log d_L}\Bigr)^3\Bigr\}
  \end{displaymath}
  otherwise \cite[§3]{bugeaud_BoundsSolutionsUnit1996}. (Lehmer's
  conjecture asserts that $\delta(L)$ can be chosen uniformly for all
  number fields~$L$). Moreover, the regulator of~$L$ is bounded above
  in terms of~$d_L$ and~$\Delta_L$. To see this, we use the main
  theorem of~\cite{sands_GeneralizationTheoremSiegel1991} and we note
  that
  \begin{enumerate}
  \item the class number of~$\Z_L$ is at least one,
  \item $L$ contains at
    most~$d_L \bigl(2+\log(d_L)/\log(2)\bigr)$ roots of unity.
  \end{enumerate}
  Therefore
  \begin{displaymath}
    \mathfrak R_L < d_L\Bigl(2+\frac{\log(d_L)}{\log(2)}\Bigr) \Bigl(\dfrac{4}{d_L-1}\Bigr)^{d_L-1} \abs{\Delta_L}^{1/2} (\log \abs{\Delta_L})^{d_L-1}.
  \end{displaymath}
  The final upper bound we obtain for the constant~$C$ in
  \cref{prop:height-reduce} grows at least linearly
  in~$\abs{\Delta_L}^{1/2}$ and exponentially in~$d_L$.
\end{rem}

\begin{cor}
  \label{cor:ideal-gen}
  Let~$C$ be as in~\cref{prop:height-reduce}. Then every principal
  ideal~$\mathfrak a$ of~$\Z_L$ admits a generator~$a\in \Z_L$ such
  that
  \begin{displaymath}
    \h(a) \leq \max \{C, \hmod(\mathfrak a)\}.
  \end{displaymath}
\end{cor}

\begin{proof}
  Apply \cref{prop:height-reduce} with~$x$ an arbitrary generator
  of~$\mathfrak a$.
\end{proof}

This corollary allows us to bound the height of a common denominator
of a given polynomial $P\in L[X]$.

\begin{prop}
  \label{prop:int-denom}
  There exists a constant~$C'$ depending only on~$L$ such that for
  every $P\in L[X]$, there exists an element~$a\in \Z_L$ such that
  $aP\in\Z_L[X]$ and $\max\{\h(a), \h(aP)\} \leq \h(P) + C'$. We can
  take
  \begin{displaymath}
    C' = \max \{C,
    \max_{\mathfrak c\in \mathfrak C}
    \hmod(\mathfrak c)\}
  \end{displaymath}
  where~$\mathfrak C$ is a set of ideals in~$\Z_L$ that are
  representatives for the class group of~$L$, and~$C$ is the constant
  from \cref{prop:height-reduce}.
\end{prop}

\begin{proof}
  Let~$\mathfrak C$ and~$C$ be as above, and let $P\in L[X]$, which we
  may assume to be nonzero.  Let
  \begin{displaymath}
    \mathfrak a = \prod_{\p\in\primes{L}}\ \p^{\max\{0, - v_\p(P)\}}
  \end{displaymath}
  be the denominator ideal of~$P$. Then
  \begin{displaymath}
    \hmod(\mathfrak a) = \sum_{\p\in\primes{L}} \frac{d_\p}{d_L} \log \max\{1,
    |P|_\p\} \leq \h(P).
  \end{displaymath}
  Let~$\mathfrak c\in \mathfrak C$ be an ideal such
  that~$\mathfrak c\mathfrak a$ is principal. By \cref{cor:ideal-gen},
  if~$C$ denotes the constant from \cref{prop:height-reduce}, we can
  find a generator~$a$ of~$\mathfrak c\mathfrak a$ such that
  \begin{displaymath}
    \h(a)\leq \max \{C,
    \hmod(\mathfrak c\mathfrak a)\} \leq \hmod(\mathfrak a)+C'\leq \h(P) + C'.
  \end{displaymath}
  Then~$aP$ has integer coefficients, and we have
  \begin{align*}
    \h(aP) & \leq \sum_{v\in\arch{L}} \dfrac{d_v}{d_L}
            \bigl(\log\max\{1,|P|_v\}  + \log\max\{1,|a|_v\} \bigr) \\
          & = \h(P) + \h(a) - \sum_{v\in\narch{L}}\dfrac{d_v}{d_L}
            \log\max\{1,|P|_v\} \\
          & = \h(P) + \h(a) - \hmod(\mathfrak a) \phantom{\dfrac{d_v}{d_L}} \\
          &\leq \h(P) + C'. \qedhere
  \end{align*}
\end{proof}

\begin{rem}
  \label{rem:CL2}
  Minkowski's bound~\cite[§V.4]{lang_AlgebraicNumberTheory1994}
  implies that we can always choose~$\mathfrak C$ in such a way that
  \begin{displaymath}
    \max_{\mathfrak c\in \mathfrak C}
    \norm(\mathfrak c) \leq \abs{\Delta_L}^{1/2} \Bigl(\frac4\pi\Bigr)^{d_L/2} \frac{d_L!}{d_L^{d_L}}.
  \end{displaymath}
  Combined with \cref{rem:CL1}, this gives an upper bound on an
  acceptable~$C'$ in \cref{prop:int-denom} depending only on~$d_L$
  and~$\Delta_L$. Under the generalized Riemann hypothesis, a much
  sharper upper bound is available: we can choose~$\mathfrak C$ in
  such a way that
  \begin{displaymath}
    \max_{\mathfrak c\in \mathfrak C}
    \norm(\mathfrak c) \leq 12 \log(\abs{\Delta_L})^2
  \end{displaymath}
  by~\cite[Thm.~3]{bach_ExplicitBoundsPrimality1990}.
\end{rem}

\section{A naive height bound for fractions}
\label{sec:frac-basic}

Let~$L$ be a number field, and let~$F\in L(X)\nzero$ be a rational
fraction of degree at most~$d\geq 1$. Write $F=P/Q$ where~$P$ and~$Q$
are coprime polynomials in~$L[X]$, and let $d_P$ and~$d_Q$ be the
degrees of~$P$ and~$Q$ respectively. Let $x_i$ for
$1\leq i\leq d_P+d_Q+1$ be distinct elements in an interval
$\Zint{A,B}\subset\Z$ that are not poles of~$F$.

We recall the interpolation algorithm to reconstruct~$F$ given the
pairs $(x_i, F(x_i))$
\cite[§5.7]{vonzurgathen_ModernComputerAlgebra2013}.
Define~$S\in L[X]$ as the polynomial of degree at
most \mbox{$d_P+d_Q$} interpolating the points~$(x_i,
F(x_i))$. Let~$a\in \Z_L$ be a common denominator for the coefficients
of~$S$, so that $T = aS$ has coefficients in~$\Z_L$. We compute the
$d_P$-th
subresultant~\cite[§3]{elkahoui_ElementaryApproachSubresultants2003}
of~$T$ and the polynomial
\begin{displaymath}
  Z = \prod_{i=1}^{d_P+d_Q+1} (X - x_i)\in\Z[X],
\end{displaymath}
which is a polynomial~$R\in \Z_L[X]$ of degree at most~$d_P$; the
usual resultant is the $0$-th subresultant. We obtain a Bézout
relation~\cite[§3.2]{elkahoui_ElementaryApproachSubresultants2003} of
the form
\begin{displaymath}
  U T + V Z = R
\end{displaymath}
where $U, V, R\in \Z_L[X]$, and moreover $\deg(U)\leq d_Q$ and
$\deg(R)\leq d_P$. Then $F = R/aU$.

In order to obtain a bound on~$\h(F)$, we first bound~$\h(S)$ using
\cref{prop:poly-basic}. Then, we use the following well-known fact
about the size of subresultants in~$\Z_L[X]$.

\begin{lem}
  \label{lem:resultant}
  Let $P, Q\in \Z_L[X]\nzero$ be polynomials of degrees~$d_P$
  and~$d_Q$ respectively, and let $0\leq k \leq \min\{d_P, d_Q\}
  -1$. Let~$R$ be the~\mbox{$k$-th} subresultant of~$P$ and~$Q$, and
  let~$U$ and~$V$ be the associated Bézout coefficients. Write
  $s = d_P + d_Q$. Then we have
  \begin{align*}
    \h(R) &\leq (d_Q - k) \h(P) + (d_P - k) \h(Q) + \dfrac{s -
    2k}{2} \log(s - 2k), \\
    \h(U) &\leq (d_Q - k - 1) \h(P) + (d_P - k) \h(Q) \\
    &\qquad + \dfrac12(s -
    2k - 1) \log(s - 2k - 1), \quad \text{and}\\
    \h(V) &\leq (d_Q - k) \h(P) + (d_P - k - 1) \h(Q) \\
    &\qquad + \dfrac12(s -
    2k - 1) \log(s - 2k - 1).
  \end{align*}
\end{lem}

For instance, \Cref{lem:resultant} allows one to bound coefficient
sizes in the subresultant version of the Euclidean algorithm
in~$\Q(X)$ \cite[§6.11]{vonzurgathen_ModernComputerAlgebra2013}.

\begin{proof}
  Let~$v\in \arch{L}$. By definition, every coefficient~$r$ of~$R$ has
  an expression as a determinant of size~$d_P+d_Q-2k$\,; its entries
  in the first~$d_Q - k$ columns are coefficients of~$P$, and its
  entries in the last~$d_P - k$ columns are coefficients of~$Q$. By
  Hadamard's
  lemma~\cite[Thm.~16.6]{vonzurgathen_ModernComputerAlgebra2013}, we
  can bound~$\abs{r}_v$ by the product of $L^2$-norms of the columns
  of this determinant in the absolute value~$v$. Hence
  \begin{displaymath}
    \abs{r}_v\leq \bigl(\sqrt{d_P + d_Q - 2k}\ \abs{P}_v \bigr)^{d_Q - k} 
    \bigl(\sqrt{d_P + d_Q
      - 2k}\ \abs{Q}_v \bigr)^{d_P - k}.
  \end{displaymath}
  Taking logarithms and summing over~$v$, we obtain the desired height
  bound on~$R$.  Similarly, the coefficients of~$U$ (resp.~$V$) are
  determinants of size \mbox{$d_P + d_Q - 2k - 1$}, where one
  column less contains coefficients of~$P$ (resp.~$Q$).
\end{proof}

\begin{prop}
  \label{prop:frac-basic-L}
  Let~$L$ be a number field, and let $\Zint{A, B}\subset \Z$. Write
  $D = B-A$ and $M = \max\{|A|,|B|\}$.  Let $F\in L(X)\nzero$ be a
  rational fraction of degree~$d\geq 1$. Let~$d_P$ and~$d_Q$ be the
  degrees of its numerator and denominator respectively. Let~$x_i$ for
  $1\leq i\leq d_P+d_Q+1$ be distinct elements of $\Zint{ A, B}$ that
  are not poles of~$F$, and assume that
  $\h(F(x_i))\leq \nolinebreak H$ for every~$i$.  Then there exist
  polynomials~$P,Q\in \Z_L[X]$ such that $F = P/Q$, $\deg P = d_P$,
  $\deg Q = d_Q$, and
  \begin{align*}
    \max\{h(P),\,h(Q)\} \leq\ &(d+1)(2d+1)H + (d+1)D\log(D) \\
    &+ (4d^2+3d)\log(2M)
    \\
    &+ (2d+2)\log(2d+1) + (d+1) C,
  \end{align*}
  where~$C$ is the constant from~\cref{prop:int-denom}.
\end{prop}

\begin{proof}
  Let~$S, a, T, Z, R, U, $ and~$V$ be as above; to choose~$a$, we use
  \cref{prop:int-denom}, so that
  \begin{displaymath}
    \max\{\h(a), \h(T)\}\leq \h(S)+C.
  \end{displaymath}
  By \cref{prop:poly-basic}, we have
  \begin{equation}
    \label{eq:frac-basic-S}    
    \h(S) \leq (2d+1)H + D\log(D) + 2d\log(2M) +
    \log(2d+1). 
  \end{equation}
  The archimedian absolute values of the coefficients of~$Z$ are
  bounded above by~$(2M)^{2d+1}$, hence
  \begin{displaymath}
    \h(Z)\leq (2d+1)\log(2M).
  \end{displaymath}
  By \cref{lem:resultant}, we have
  \begin{align*}
    \h(R) &\leq (d+1)\h(T) + d(2d+1)\log(2M) +
           \dfrac{2d+1}{2}\log(2d+1), \quad \text{and} \\
    \h(U) &\leq d \h(T) + d(2d+1) \log (2M) + d \log(2d+1).
  \end{align*}
  Then $F = R/aU$, and
  \begin{align*}
    \max\{\h(R),\h(aU)\} &\leq \max\{\h(R), \h(a)+\h(U)\} \\
                         &\leq (d+1)(\h(S)+C) + d(2d+1) \log(2M)\\
    &\qquad + \frac{2d+1}{2}\log(2d+1).
  \end{align*}
  Using the upper bound~\eqref{eq:frac-basic-S} on~$\h(S)$ ends the
  proof.
\end{proof}

The bound we obtain on~$\h(F)$ in \cref{prop:frac-basic-L} is
roughly~$O(d^2H)$. This motivates a result like~\cref{thm:main-frac},
where the dependency on~$H$ is only linear.

\section{Preparations for the proof of~\texorpdfstring{\cref{thm:main-frac}}{Theorem~1.2}}
\label{sec:frac-lem}

In this section, we state preparatory lemmas for the proof of
\cref{thm:main-frac}; the reader might wish to skip them until their
use in the proof becomes apparent.

We keep the notation introduced at the beginning
of~§\ref{sec:heights}, to which we add the following.
If~$\p\in \primes{L}$, we denote by~$v_\p$ the $\p$-adic valuation
on~$L$, with the convention that~$v_\p(0) = +\infty$. When
considering~$\p$ as a finite place of~$L$, we write $\abs{\cdot}_\p$
for the associated absolute value. We denote by~$d_\p$ and~$e_\p$ the
local degree and ramification index of~$\p$ in the
extension~$L/\Q$. With our normalizations, the following formula holds
for every~$x\in L$ and~$\p\in \primes{L}$:
\begin{displaymath}
  \abs{x}_\p = \norm(\p)^{-v_\p(x)/d_\p}.
\end{displaymath}
Finally, for~$r\in \R$, we denote the upper integral part of~$r$ by~$\ceil{r}$.

\begin{lem}
  \label{lem:subinterval}
  Let $\Zint{A,B}\subset\Z$, let $D = B-A$, and let $\eta\geq 1$;
  assume that~$D\geq 2\eta$. Let~$S$ be a subset of~$\Zint{ A, B} $
  containing at least~$D/\eta$ elements, and let\,
  $1\leq k\leq \frac{D}{2\eta}$ be an integer. Then there exists a
  subinterval of $\Zint{A, B}$ of length at most~$\ceil{2\eta k}$
  containing at least $k+1$ elements of $S$.
\end{lem}

\begin{proof}
  Let~$m\in \Z$ such that $m\geq 1$. Then for each~$n\geq 1$, the
  following intervals of~$\Z$:
  \begin{displaymath}
    \Zint{0,m},\ \Zint{m+1,2m+1},\ \ldots,\ \Zint{(n-1)(m+1), n(m+1)-1}
  \end{displaymath}
  form a partition of~$\Zint{0, n(m+1)-1}$ in~$n$ intervals of
  length~$m$. Taking \mbox{$m = \ceil{2\eta k}$}
  and~$n = \ceil{D/(2\eta k)}$, the right endpoint of the latter
  interval is at least~$D$. Therefore, by translating the above
  partition and intersecting it with~$\Zint{A,B}$, we obtain a
  partition of~$\Zint{A, B}$ in at most $\ceil{D/(2\eta k)}$ intervals
  of length at most $\ceil{2\eta k}$. In the case that each of these
  intervals contains at most~$k$ elements of~$S$, we deduce that
  \begin{displaymath}
    \frac{D}{\eta} \leq \#S \leq k \ceil{\dfrac{D}{2\eta k}} < 
      \dfrac{D}{2\eta} + k.
  \end{displaymath}
  This is absurd because $k \leq \frac{D}{2\eta}$.
\end{proof}

\begin{lem}
  \label{lem:sum-primes-L}
  Let $R\in \Z_L\nzero$ be a non-unit. Then
  \begin{displaymath}
    \sum_{\substack{\p\in \primes{L},\,\p|R \\ \p|p\in
        \primes{\Q}}} \dfrac{e_\p\log (\norm(\p))}{p-1} \leq d_L (2\log \log \abs{\norm(R)} + 4).
  \end{displaymath}
\end{lem}

\begin{proof}
  First, we assume that $L=\Q$, so that $R\in\Z$ and~$\abs{R}\geq 2$.
  Let~$m$ be the number of prime factors in~$R$, and let~$(p_i)$ be
  the sequence of prime numbers in increasing order. It is enough to
  prove the claim for the integer~$R' = \prod_{i=1}^m p_i$, which has
  both a greater left hand side, since~$\log(p)/(p-1)$ is a decreasing
  function of~$p$, and a smaller right hand side, since
  $R'\leq \abs{R}$. We can assume that $m\geq 2$. Then
  \begin{displaymath}
    \sum_{i=1}^m \dfrac{\log(p_i)}{p_i - 1} =
    \sum_{i=1}^m \dfrac{\log(p_i)}{p_i} + \sum_{i=1}^m \dfrac{\log(p_i)}{p_i(p_i-1)}
    \leq \log(p_m) + 3
  \end{displaymath}
  by Mertens's first theorem
  \cite{mertens_BeitragZurAnalytischen1874}, and because the sum of
  the second series is less than~$0.76$.
  By~\cite{rosser_ExplicitBoundsFunctions1941}, we have
  $p_m < m\log m + m\log \log m$ if $m\geq 6$; thus the rough bound
  $p_m \leq m^2$ holds. Since $m\leq \log(R')/\log(2)$, the result in
  the case~$L=\Q$ follows.
  
  In the general case, if $\p|R$ lies above~$p$, then~$p$
  divides~$\norm(R)$, and $\abs{\norm(R)}\geq 2$. We apply
  \cref{lem:sum-primes-L} to~$\norm(R)\in \Z$: hence
  \begin{align*}
    \sum_{\p|R} \dfrac{e_\p\log(\norm(\p))}{p-1}
    &\leq \sum_{p|N_{L/\Q}(R)}
      \dfrac{\sum_{\p|p} e_\p
      \log(\norm(\p))}{p-1} \\
    &=
      d_L\sum_{p|N_{L/\Q}(R)}
      \dfrac{\log(p)}{p-1} \\
    &\leq d_L (2\log\log\abs{\norm(R)}+ 4). \qedhere
  \end{align*}
\end{proof}

\begin{lem}
  \label{lem:padic-L}
  Let $\p\in \primes{L}$ be a prime ideal lying over
  $p\in\primes{\Q}$, and let~$L_\p$ be the~$\p$-adic completion
  of~$L$. Let $Q\in L_\p[X]$ be a polynomial of degree~$d\geq 0$, and
  assume that $v_\p(Q) = 0$. Let $x_1,\ldots, x_n$ be distinct values
  in $\Zint{A,B}$, and write $D = B-A$; assume that $D\geq 1$. Let
  $\beta\in \N$. Then
  \begin{equation}
    \label{eq:padic}
    \sum_{i=1}^n \min\{\beta, v_\p(Q(x_i))\} \leq
    d \left(\beta +  \dfrac{d_\p\log(D)}{\log \norm(\p)} + \dfrac{e_\p D}{p-1}\right).
  \end{equation}
\end{lem}

\begin{proof}
  We can assume that~$d\geq 1$.  Let~$\lambda$ be the leading
  coefficient of~$Q$, and let $\alpha_1,\ldots,\alpha_{d}$ be the
  roots of~$Q$ in an algebraic closure of~$L_\p$, where we
  extend~$\abs{\cdot}_\p$ and~$v_\p$. Up to reindexation, we may assume
  that $\abs{\alpha_j}_\p \leq 1$ for $1\leq j\leq t$, and
  $\abs{\alpha_j}_\p > 1$ for $t+1\leq j\leq d$. For every~$i$, we
  have
  \begin{displaymath}
    \abs{{Q}(x_i)}_{\p} = \abs{\lambda}_\p \prod_{i=1}^{d} \abs{x_i -
      \alpha_j}_{\p} =
      \biggl(\abs{\lambda}_\p \prod_{j = t+1}^{d}
    \abs{\alpha_j}_\p \biggr) \prod_{j=1}^t \abs{{x_i - \alpha_j}}_{\p}.
  \end{displaymath}
  Since~$v_\p(Q)=0$, we have
  \begin{displaymath}
    \biggl(\abs{\lambda}_\p \prod_{j = t+1}^{d}
    \abs{\alpha_j}_\p \biggr) = 1.
  \end{displaymath}
  Therefore, for each~$1\leq i\leq n$,
  \begin{displaymath}
    v_\p({Q}(x_i)) =
    \sum_{j=1}^t v_\p(x_i -
    \alpha_j). 
  \end{displaymath}

  Let $k\in\N$ be such that $p^{k} \leq D < p^{k+1}$.  Since the~$x_i$
  are all distinct modulo~$p^{k+1}$, there exist at most~$d$ values
  of~$i$ such that $v_\p(x_i - \alpha_j) > k e_\p$ for some~$j$. For
  these indices~$i$, we bound $\min\{\beta, v_\p(Q(x_i))\}$ from above
  by~$\beta$. This accounts for the term~$d\beta$ in
  inequality~\eqref{eq:padic}.

  For all other values of~$i$ (say~$i\in I$), we have
  $v_\p(x_i - \alpha_j) \leq ke_\p$ for every $1\leq j\leq t$. For
  each~$1\leq w\leq ke_\p$ and $1\leq j\leq t$, define
  \begin{displaymath}
    S_{j,w} = \{i \in I\st v_\p(x_i-\alpha_j) \geq w\}.
  \end{displaymath}
  For fixed~$j$ and~$w$, all the values~$x_i$ for~$i\in S_{j,w}$
  coincide modulo $p^{\ceil{w/e_\p}}$, so
  \begin{displaymath}
    \# S_{j,w} \leq \ceil{\frac{D}{p^{\ceil{w/e_\p}}}}.
  \end{displaymath}
  Note that for all~$i\in I$ and~$1\leq j\leq t$, the number of
  values of $w\in \Zint{1, ke_\p}$ such that $i\in S_{j,w}$ is
  precisely~$v_\p(x_i-\alpha_j)$. Therefore,
  \begin{align*}
    \sum_{i\in I} v_\p({Q}(x_i))
    &= \sum_{i\in I} \sum_{j=1}^t
      v_\p(x_i - \alpha_j) \\
    &= \sum_{j=1}^t \sum_{w=1}^{ke_\p} \# S_{j,w} \\
    &\leq  d\sum_{w=1}^{k e_\p}\Bigl(\frac{D}{p^{\ceil{w/e_\p}}}+1\Bigr)\\
    &= d e_\p \sum_{w=1}^k \Bigl(\frac{D}{p^w}+1\Bigr) \\
    &\leq d e_\p k + \frac{d e_\p D}{p-1}.
  \end{align*}
  Since
  \begin{displaymath}
    k\leq \frac{\log(D)}{\log(p)} = \frac{d_\p}{e_\p}\cdot \frac{\log(D)}{\log\norm(\p)},
  \end{displaymath}
  this accounts for the two remaining terms in
  inequality~\eqref{eq:padic}.
\end{proof}

\section{Heights of fractions from their values}
\label{sec:frac-proof}

This final section is devoted to the proof of~\cref{thm:main-frac} and
its corollary. We keep the notation from~§\ref{sec:frac-lem}, and recall
the main statement for the reader's convenience.

\begin{thm}
  \label{thm:main-frac-L}
  \mainstatement
\end{thm}

\begin{proof}
  We can assume that~$F\neq 0$. We have~$D\geq 4\eta d$, so by
  \cref{lem:subinterval} with $k = 2d$, we can find a subinterval of
  $\Zint{ A, B}$ of length at most~$\ceil{4\eta d}$ containing~$2d+1$
  elements of~$S$, denoted by~$x_1,\dots,x_{2d+1}$.  We use
  these~$x_i$ as evaluation points to apply \cref{prop:frac-basic-L}:
  we can write $F = P/Q$ where $P,Q\in \Z_L[X]$ are coprime in~$L[X]$
  and satisfy
  \begin{align*}
    \max\{\h(P), \h(Q)\} &\leq (d+1)(2d+1) H + 2d\ceil{4\eta d} \log(\ceil{4\eta d}) \\
                       &\quad + (4d^2+3d)\log(2M) + (2d+2)\log(2d+1) \\
                       &\quad + (d+1)C_1\\
                       &\leq (27+C_1)\eta d^2H,
  \end{align*}
  where~$C_1$ is the constant from \cref{prop:height-reduce}. To
  simplify the right hand side, we use the inequalities $1\leq d$,\;
  $1\leq \eta$,\; $\ceil{4\eta d} \leq D\leq 2M$,\;
  \mbox{$\ceil{4\eta d} \leq 5\eta d$}, and \mbox{$\log(2M)\leq H$}.

  Let $x\in S$. We define ideals~$\mathfrak s_x$, $\mathfrak n_x$
  and $\mathfrak d_x$ of~$\Z_L$ as follows:
  \begin{displaymath}
    \mathfrak s_x = \gcd\bigl((P(x)), (Q(x))\bigr), \quad 
    (P(x)) =  \mathfrak n_x  \mathfrak s_x, \quad (Q(x)) =
    \mathfrak d_x  \mathfrak s_x.
  \end{displaymath}
  Then $(F(x)) = \mathfrak n_x \mathfrak d_x^{-1}$. The
  ideal~$\mathfrak s_x$ encodes the simplifications that occur when
  evaluating~$P/Q$ at~$x$. The heart of the proof is to show
  that~$\mathfrak s_x$ has small norm for at least some values
  of~$x$. Let~$\mathfrak r$ be the greatest common divisor of all the
  coefficients of~$P$ and~$Q$.

  \begin{claim}
    \label{claim:small-sx}
    There exist at least $2 d d_L+1$ elements~$x$ of~$S$ such that
    \begin{displaymath}
      \hmod(\mathfrak s_x) \leq \hmod(\mathfrak r) + C\eta d\log(\eta d H)
    \end{displaymath}
    for some constant~$C$ depending only on~$L$.
  \end{claim}

  Let us explain how to finish the proof assuming that
  \cref{claim:small-sx} holds. By \cref{lem:bad-values}, we can find
  an~$x\in S$ among these $2 d d_L+1$ values such that for every
  $v\in \arch{L}$, we have
  \begin{displaymath}
    \abs{P(x)}_v
    \geq \dfrac{\abs{P}_v}{(2M)^d (d+1)}\quad\text{and}\quad
    \abs{Q(x)}_v
    \geq \dfrac{\abs{Q}_v}{(2M)^d (d+1)}.
  \end{displaymath}
  Then, by \cref{def:heights}, we have
  \begin{align*}
    \h(F) &=\sum_{v\in \arch{L}} \frac{d_v}{d_L} \log \max\{\abs{P}_v, \abs{Q}_v\}
            - \hmod(\mathfrak r)\\
          &\leq \sum_{v\in \arch{L}} \frac{d_v}{d_L} \log\max\{\abs{P(x)}_v, \abs{Q(x)}_v\}
            - \hmod(\mathfrak r) \\
    &\qquad + d\log(2M)+\log(d+1) \\
          &\leq \sum_{v\in \places{L}} \frac{d_v}{d_L} \log\max\{\abs{P(x)}_v, \abs{Q(x)}_v\}
            + \hmod(\mathfrak s_x) - \hmod(\mathfrak r)\\
    &\qquad + d\log(2M)+\log(d+1)\\
          & \leq H + C \eta d \log(\eta d H) + d\log(2M) + \log(d+1),
  \end{align*}
  as claimed.

  In order to prove \cref{claim:small-sx}, a crucial remark is
  that~$\mathfrak s_x$ divides the resultant~$R$ of~$P$ and~$Q$.
  By \cref{lem:resultant}, we have
  \begin{displaymath}
    \h(R)\leq d \h(P) + d \h(Q) + d\log(2d) \leq (55+2C_1)\eta d^3H.
  \end{displaymath}
  Let $\p\in\primes{L}$ be a prime factor of~$R$ with
  valuation~$\beta_\p$, and let~$I$ be a subset of~$S$ with~$n$
  elements. We claim:
  \begin{equation}
    \label{eq:frac-proof-vp}
    \sum_{x\in I} v_\p(\mathfrak s_x) \leq n\, v_\p(\mathfrak r) +
    d\left(\beta_\p + \dfrac{d_\p\log(D)}{\log \norm(\p)} + \dfrac{e_\p D}{p-1}\right).
  \end{equation}
  
  To prove~\eqref{eq:frac-proof-vp}, we can work in the $\p$-adic
  completion~$L_\p$ of~$L$. Let~$\pi$ be a uniformizer of~$L_\p$, and
  let $r = \min\{v_\p(P),v_\p(Q)\}$ be the $\p$-adic valuation
  of~$\mathfrak r$.  Write $P_1 = P/\pi^r$, $Q_1 = Q/\pi^r$. Then one
  of~$P_1$ and~$Q_1$ is not divisible by~$\pi$; for instance, assume
  that~$\pi$ does not divide~$Q_1$. Then, for every $x\in S$,
  \begin{displaymath}
    v_\p(\mathfrak s_x)\leq \min \bigl\{\beta_\p, v_\p(Q(x))\}\leq v_\p(\mathfrak r) +
    \min \bigl\{\beta_\p, v_\p(Q_1(x)) \bigr\}.
  \end{displaymath}
  Therefore inequality~\eqref{eq:frac-proof-vp} follows from \cref{lem:padic-L}.

  Inequality~\eqref{eq:frac-proof-vp} gives an upper bound on
  the~$\p$-adic valuation of the ideal $\prod_{x\in I} \mathfrak s_x$.
  Taking the product over the prime factors~$\p$ of~$R$, we obtain an
  upper bound on the norm of that ideal. We can assume that~$R$ is not
  a unit, otherwise \cref{claim:small-sx} holds trivially. We obtain
  \begin{displaymath}
    \begin{aligned}
      \abs{\prod_{x\in I} \norm(\mathfrak s_x)} & \leq \norm(\mathfrak
      r)^n \abs{\norm(R)}^d \\
      &\qquad \cdot \exp \biggl(\
      \sum_{\substack{\p\in\primes{L},\,\p|R\\ \p|p\in\primes{\Q}}}
      \Bigl( d d_\p \log(D) + d D\dfrac{e_\p\log \norm(\p)}{p-1}
      \Bigr) \biggr)\\
      &\leq \norm(\mathfrak r)^n \abs{\norm(R)}^d \\
      &\qquad \cdot \exp \bigl(d d_L
      \log(D)\log \abs{\norm(R)}/\log(2) \\
      &\qquad \qquad \quad + d d_L D (2\log\log\abs{\norm(R)} + 4) \bigr).
    \end{aligned}
  \end{displaymath}
  Indeed,~$R$ has at most $\log\abs{\norm(R)}/\log(2)$ prime factors, and we
  can apply \cref{lem:sum-primes-L}.  Since
  $\hmod(R) \leq (55+2C_1)\eta d^3 H$, we obtain
  \begin{align*}
    \sum_{x\in I} \hmod(\mathfrak s_x)
    &\leq n \hmod(\mathfrak r) + d\hmod(R)
      + d  d_L\frac{\log(D)}{\log(2)} \hmod(R)\\
    &\qquad
      + d D (2 \log \log \abs{\norm(R)} + 4)
    \\
    &\leq n\hmod(\mathfrak r) + C_2 \bigl(\eta d^4 H \log(D) + d D \log (\eta d H)\bigr)
  \end{align*}
  with
  \begin{equation}
    \label{eq:C2}
    C_2 =  \max\left\{\frac{3d_L (55+2C_1)}{2\log(2)}, 10 + 2\log (d_L) + 2\log(55+2C_1)  \right\}.
  \end{equation}
  Here we use that~$\log(\eta d H)\geq 1$, and $\log(D)\geq 2\log 2$.

  Now we put into play our assumptions about~$D$ and~$S$ being
  sufficiently large.  Since $D\geq \eta d^3 H \geq 4 > \exp(1)$, and
  the function~$t/\log(t)$ is increasing for $t>\exp(1)$, we have
  \begin{displaymath}
    \dfrac{D}{\log(D)} \geq \dfrac{\eta d^3 H}{3\log (\eta dH)}.
  \end{displaymath}
  Moreover,
  \begin{displaymath}
    \#S - 2 d d_L \geq \frac{D}{\eta} - \dfrac{D}{2\eta} = \dfrac{D}{2\eta}.
  \end{displaymath}
  Therefore,
  \begin{align*}
    \sum_{x\in I} \hmod(\mathfrak s_x)
    &\leq n\hmod(\mathfrak r) + 4C_2 d D \log(\eta d H) \\
    &\leq n\hmod(\mathfrak r) + 8C_2 \eta d \log(\eta d H) (\#S - 2 d d_L).
  \end{align*}

  This shows that in every subset of~$\#S - 2d d_L$ elements of~$S$,
  at least one satisfies the upper bound
  $\hmod(\mathfrak s_x)\leq \hmod(\mathfrak r) + 8C_2\eta d \log(\eta
  dH)$. Hence \cref{claim:small-sx} holds with $C=8C_2$, so the
  theorem holds with~$C_L = 8C_2$.

  In general,~$C_2$ is defined in~\eqref{eq:C2}; in this
  equation,~$C_1$ is a constant such that \cref{prop:int-denom}
  holds. By \cref{rem:CL1,rem:CL2},~$C_1$ can be bounded above
  explicitly in terms of~$d_L$ and~$\Delta_L$ only, so the same
  property holds for~$C_L$.  If~$L=\Q$, we have $C_1=0$, so we can
  take $C_2=120$.
\end{proof}

To conclude, we give the proof of \cref{cor:main}.

\begin{cor}
  \label{cor:main-proved}
  \corstatement
\end{cor}

\begin{proof}
  We want to apply \cref{thm:main-frac} on an interval of the form
  $\Zint{0,D}$ for some integer $D\geq 4d$, with~$\eta=2$
  and~$S = \Zint{0,D}\backslash V$. The set~$S$ contains at least
  $D/\eta$ elements as soon as $D\geq 2\# V$.

  For every $x\in S$, we have $\h(x)\leq \log(D)$, hence
  \begin{displaymath}
    \h(F(x))\leq c\max\{1, d\log d + d\log D\}.
  \end{displaymath}
  Hence, if we let
  \begin{displaymath}
    \bigH(D) = \max\{4, \log(2D), c(d\log d + d\log D)\}
  \end{displaymath}
  we can apply \cref{thm:main-frac} with $H = \bigH(D)$ as soon as the condition
  \begin{displaymath}
    D \geq 2d^3 \bigH(D)
  \end{displaymath}
  holds. We check that we can choose
  \begin{displaymath}
    D = \max\{2\#V, \ceil{4cd^4\log(4cd^4)}\}.
  \end{displaymath}
  Then, \cref{thm:main-frac} yields
  \begin{align*}
    \h(F)&\leq \bigH(D) + 1920d\log(2d\bigH(D)) + d\log(2D) + \log(d+1).
  \end{align*}
  We have~$\bigH(D)\leq 4cd\log(dD)$ and $2d \bigH(D)\leq D$, hence
  \begin{align*}
    \h(F) &\leq 4cd\log(dD) + 1920 d\log(D) + d\log(2D) + \log(d+1)\\
          &\leq (4c + 1923)d\log(dD) \\
          &\leq (4c+1923)d (\log(2 d \max\{1,\#V\}) + \log(5cd^5\log(4cd^4)))
  \end{align*}
  To simplify this expression further, we write
  \begin{displaymath}
    \log(5cd^5\log(4cd^4))\leq \log(20c^2d^9)\leq 3 + 2\log(c) + 9\log(d).
  \end{displaymath}
  hence, after other simplifications,
  \begin{displaymath}
    \h(F) \leq C d\log(4d)
  \end{displaymath}
  with
  \begin{displaymath}
    C= (4c+1923) (12 + \log\max\{1,\#V\} + 2\log(c)),
  \end{displaymath}
  as claimed.
\end{proof}

\bibliographystyle{abbrv}
\bibliography{heights-interpolation.bib,unpublished.bib}

\end{document}